\documentclass[11pt,article,reqno]{amsart}
\usepackage{amsmath, amsthm, amsfonts}

\usepackage[top=25mm,bottom=25mm,left=27mm,right=27mm]{geometry}

\usepackage{color}

\makeatletter
\def\section{\@startsection{section}{1}%
\z@{1\linespacing\@plus\linespacing}{1\linespacing}%
{\bf\centering}}
\def\subsection{\@startsection{subsection}{0}%
\z@{\linespacing\@plus\linespacing}{\linespacing}%
{\bf}}
\def\subsubsection{\@startsection{subsubsection}{0}%
\z@{\linespacing\@plus\linespacing}{\linespacing}%
{\bf}}
\makeatother

\makeatletter
\@addtoreset{equation}{section}
\makeatother

\usepackage{graphicx}
\usepackage{amsmath, amsfonts}
\usepackage{color}  
\newtheorem{theorem}{Theorem}[section]

\newtheorem{lemma}[theorem]{Lemma}
\newtheorem{proposition}[theorem]{Proposition}
\theoremstyle{definition}

\theoremstyle{remark}
\newtheorem{remark}[theorem]{Remark}

\def\RR{\mathbb{R}}

\def\PP{\mathbb{P}}
\def\EE{\mathbb{E}}

\def\al{\alpha}
\def\be{\beta}

\def\de{\delta}

\def\ep{\varepsilon}
\def\ph{\varphi}
\def\la{\lambda}
\def\sk{\smallskip}

\def\ds{\displaystyle}
\def\cH{\mathcal{H}} 
\def\rd{\mathrm{d}}  

\newcommand{\bra}[1]{\left\lbrace#1\right\rbrace}
\newcommand{\bkt}[1]{\left[#1\right]}

\def\cM{M} 
\def\rr{\mathbb{R}}
\def\cF{\mathcal{F}}
\def\pp{\PP}
\def\s{\mathbb{S}}
\def\wn{\widetilde{N}}
\def\sm{{s-}}
\def\cL{\mathcal{L}}
\def\indiq{1}
\def\cD{\mathcal{D}}
\def \llb {\left\lbrace}
\def \rrb {\right\rbrace}
\def\cP{\mathcal{P}}
\def \( {\left(}
\def \) {\right)}
\def\nn{\mathbb{N}}
\def\E{\mathbb{E}}

\newenvironment{preuve}{\noindent {\it Proof}}{\hfill$\square$}
\def\e{\varepsilon}
\def \lba {\left|}
\def \rba {\right|}
\def \[ {\left\lbrack}
\def \] {\right\rbrack}
\def\cG{\mathcal{G}}

\begin{document}
\title[Range and graph of stable-like processes]{Hausdorff dimension of the range and the graph of stable-like processes}
\author[X. Yang]{Xiaochuan Yang}

\address{Dept. Statistics \& Probability, Michigan State University, 48824 East Lansing, MI, USA}
\address{Universit\'e Paris-Est, LAMA(UMR8050), UPEMLV, UPEC, CNRS, F-94010 Cr\'eteil, France.}

\email{yangxi43@stt.msu.edu, xiaochuan.j.yang@gmail.com}

\thanks{
 \\ \medskip
\noindent
2010 {\it MS Classification}:  60H10 \and 60J25 \and 60J75 \and 28A78. \\
\noindent
}

\begin{abstract}
We determine the Hausdorff dimension for the range of a class of pure jump Markov processes in $\mathbb{R}^d$, which turns out to be random and depends on the trajectories of these processes. The key argument is carried out through the SDE representation of these processes.  The method developed here also allows to compute the Hausdorff dimension for the graph.
\end{abstract}

\maketitle

\baselineskip 0.5 cm

\bigskip \medskip

\section{Introduction}

The range of various stochastic processes provides interesting examples of random fractals.  The determination of their Hausdorff dimension is a natural question.  
For L\'evy processes, this question has been addressed by different authors, see for instance Taylor \cite{taylor1953},  Blumenthal and Getoor \cite{blumenthal1960stable}, McKean \cite{mckean1955stable}, Pruitt \cite{pruitt1969range}, Khoshnevisan, Xiao and Zhong \cite{khoshnevisan2003range}, and the survey article by Xiao \cite{xiao2004survey} for an exhaustive list of literature on this topic. In particular, Pruitt \cite{pruitt1969range} characterized the Hausdorff dimension for the range of general L\'evy processes in terms of their potential operators, while  Khoshnevisan, Xiao and Zhong \cite{khoshnevisan2003range} measured their range in terms of the characteristic exponent.

Recently, there has been much interest in understanding Markov processes generated by pseudo-differential operators, we refer the readers to the monograph by Jacob \cite{jacob2005vol3} and a recent survey book by B\"ottcher, Schilling and Wang \cite{bottcher2013survey} (we adopt the terminology "L\'evy-type processes" therein). These processes are usually spatially inhomogeneous which is an important feature because real life data (e.g. financial, geographical and meteorologic data) which have been modeled by L\'evy processes often exhibit different characteristics in different locations. Therefore, modeling with L\'evy-type processes can be relevant.

 The determination of the Hausdorff dimension of L\'evy-type processes, contrary to the L\'evy case, seems to be far from being accomplished. In particular, only upper (see Schilling \cite{schilling1998rangefeller}) and lower bounds (see Knopova, Schilling and Wang \cite{knopova2015range}) are known under various conditions and these bounds do not match in general.

The purpose of this article is to find the exact Hausdorff dimension of the sample paths of a specific class of L\'evy-type processes in $\RR^d$, called stable-like processes, whose generator can be written for all $f\in C^2_c(\RR^d)$, twice continuously differentiable functions with compact support, 
\begin{equation}\label{eq: generator}
\mathcal{L}^{{\be}} f(x) = \int \bkt{ f(x+u)-f(x)- 1_{|u|\le 1} u \cdot \nabla f(x) } {\be(x)|u|^{-d-\be(x)}\rd u},
\end{equation}
where $|\cdot|$ is the Euclidean norm in $\RR^d$ and $\be$ is a Lipschitz map from $\RR^d$ to a compact subset of $(0,2)$. {This function $\be$ is the key which gives all the information on dimensional properties of stable-like processes.}   The uniqueness in law of a Markov process with generator \eqref{eq: generator} was proved by Bass \cite[page 274]{bass1988uniqueness}. Actually, Bass showed the uniqueness in law for a large class of L\'evy-type operators under quite weak regularity conditions on the jump kernel (Dini-continuity for $\be$ in the stable-like case). Moreover, stable-like processes are Feller processes \cite[page 285]{bass1988uniqueness}, so the strong Markov property holds.  Here, we assume stronger regularity condition (Lipschitz continuity) because we will use a jump SDE representation of stable-like processes (especially the pathwise uniqueness). We refer the readers to the monograph by Kolokoltsov  \cite[Chapter 7]{kolokoltsov2011book} for more on  distributional properties of stable-like processes, e.g. heat kernel estimates. 

\medskip
The main result of this paper is the following. 

\begin{theorem}\label{theorange} Let $\cM$ be a stable-like process in $\rr^d$, that is a Markov process with generator \eqref{eq: generator}. Then a.s. for every open interval $I= (a,b)\subset \RR^+$, $$\dim_\cH \big(\cM(I)\big) = d\wedge \sup_{s\in I}\be(\cM_s).$$
Here and after, $\dim_\cH E$ denotes the Hausdorff dimension of the set $E$.
\end{theorem}

Let us comment the proof. To get the upper bound, we combine the classical variation methods with the "slicing" technique introduced in \cite{yang2015jumpdiffusion}, which  allows to distinguish different local behavior of $\cM$, see Section 3. To get the lower bound, the strategy is to couple our process with a family of other stable-like processes whose Hausdorff dimension {is} known, then compare the Hausdorff dimension of their sample paths with ours using pathwise uniqueness and the Markov property, see Section 4. 

\sk

With  the same strategy, we are also able to compute the Hausdorff dimension of the graph of stable-like processes. 
\begin{theorem}\label{theograph} Let $\cM$ be a stable-like process as in Theorem \ref{theorange}. Let $Gr_{I}(\cM)= \{ (t,\cM_t) : t\in I\}$ be the graph of $\cM$ on the interval $I \subset\rr^+$.
\begin{enumerate}
\item If $d\ge 2$, then a.s. for every open interval $I\subset\rr^+$, 
$$\dim_\cH \Big(Gr_I(\cM)\Big) = 1 \vee \sup_{t\in I}\be(\cM_t).$$ 
\item If $d=1$, then a.s. for every open interval $I\subset\rr^+$, 
\begin{align}\label{eqinthm1.2}
 \dim_\cH \Big(Gr_I(\cM)\Big) = 1\vee \Big(2-\frac{1}{\sup_{t\in I}\be(\cM_t)}\Big).
\end{align}
\end{enumerate} 
\end{theorem}
This theorem generalizes classical results \cite{blumenthal1962zeroANDgraph,jain1968graph} on the Hausdorff dimension for the graph of $\al$-stable processes in $\RR^d$. 
Historically, Blumenthal and Getoor \cite{blumenthal1962zeroANDgraph} treated the recurrent case  ($d=1$ and $\al>1$), and Jain and Pruitt \cite{jain1968graph} the transient case ($\al<d$).
Later, Pruitt and Taylor \cite{pruitt1969stablecomponents} investigated, among others things, the asymptotic behavior of the sojourn time of a L\'evy process with stable components and related the exact Hausdorff measure of the graph of such process to these results. We follow and adapt, when necessary, the arguments of Pruitt and Taylor  \cite{pruitt1969stablecomponents}. 

This paper is organized as follows. We first recall some basic properties of the  stable-like processes in Section 2. We study the $p$-variation of $\cM$ in Section 3 to yield the upper bound for the dimension of the range of stable-like processes. The lower bound is proved in Section 4 using a coupling argument. Finally, we deal with the dimension of the graph of $\cM$ (Theorem \ref{theograph}) in Section 5.

In the whole paper, $C$ is a positive finite constant independent of the problem, that may change from line to line.

From now on, we only consider the time interval $[0,1]$, extension to any interval is straightforward.


\section{Preliminaries}

First let us introduce the SDE with jumps associated with stable-like processes.  Let  $(\Omega, \cF,(\cF_t), \pp)$ be  a filtered probability space satisfying the usual conditions. Let $\la$ be the Lebesgue measure on $\rr^+$, $H$ be the uniform probability measure on $\mathbb{S}^{d-1}$ and $\pi(dr)= r^{-2}\rd r$ on $\rr^+$.  Denote by $N$ a Poisson random measure on the product space $\rr^+\times\s^{d-1}\times\rr^+$ adapted to the filtration $(\cF_t)$ and with intensity $\la\otimes H\otimes\pi$. We denote by $\wn$ the corresponding compensated Poisson measure.

\begin{proposition}\label{proexistence} Let $\be$ be as in \eqref{eq: generator}. For every $\cF_0$-measurable random variable $\cM_0$, there exists a unique pathwise solution to the stochastic differential equation,
\begin{multline}\label{theequation} \cM_t = \cM_0 + \int_0^t\int_{S^{d-1}}\int_0^1 \theta r^{1/\be(\cM_\sm)} \wn(ds,d\theta,\rd r) \\ 
+ \int_0^t \int_{\s^d} \int_1^{+\infty} \theta r^{1/\be(\cM_{s-})} N(\rd s,\rd \theta, \rd r).
\end{multline}
Furthermore, the solution to \eqref{theequation} is a c\`adl\`ag $(\cF_t)$-adapted Feller  process whose generator is $\cL^\be$ in \eqref{eq: generator}.  
\end{proposition}  

\begin{remark} This SDE representation for stable-like processes was first proved in  \cite[page 111]{tsuchiya1992}, we include a proof for completeness.
\end{remark}

\begin{proof}
Let us first consider the well-posedness of the SDE.  By an interlacement procedure for the non compensated Poisson integral in \eqref{theequation} (see \cite[Proposition 2.4]{fu2010positiveSDE}), it is enough to prove that the following SDE has a unique pathwise solution:
\begin{equation}
X_t = X_0 + \int_0^t\int_{S^{d-1}}\int_0^1 \theta r^{1/\be(X_\sm)} \wn(\rd s,\rd \theta,\rd r).
\end{equation}
Classical Picard iteration, Gronwall's lemma and localization procedure entails the existence of a unique pathwise solution once we check the usual (local) Lipschitz continuity and linear growth condition on the coefficients of the SDE, see for instance \cite[Section 3.1]{situ2005}. In other words, it suffices to check that there exists a positive finite constant $C$ such that for all $x,y\in\rr^d$,  
\begin{align*}
 \int_{\s^{d-1}}\int_0^{1} |\theta r^{1/\be(x)}|^2 \, r^{-2}\rd r \,H(\rd\theta) \le C(1+|x|^2), \\
 \int_{\s^{d-1}}\int_0^{1} |\theta r^{1/\be(x)}-\theta r^{1/\be(y)}|^2 \,  r^{-2}\rd r\,H(\rd \theta) \le C|x-y|^2.
\end{align*}
Actually, the first integral is bounded from above uniformly in $x$, that is 
\begin{align}\label{eq: linear growth+}
 \int_{\s^{d-1}}\int_0^{1} |\theta r^{1/\be(x)}|^2 \, r^{-2}\rd r \,H(\rd\theta) \le C
\end{align} 
These conditions are checked in Appendix. 

To prove the second statement, one starts with the observation that
\begin{align*}
\cL^\be f(x)= \int_{\s^{d-1}}\int_{\rr^+} \Big( f(x+\theta r^{1/\be(x)}) -f(x) - \indiq_{0<r<1} r^{1/\be(x)}\theta \cdot\nabla f(x)\Big) \,\frac{\rd r}{r^2} H(\rd \theta),
\end{align*}
where a change of variable $u=\theta r^{1/\be(x)}$ was used for all $x\in\RR^d$ in \eqref{eq: generator}. Applying It\^o's formula to $M$ the solution of \eqref{theequation}, for all $f\in C^2_c(\RR^d)$, 
\begin{multline*}
f(M_t) - f(M_0) - \int_0^t \cL^\be f(M_s)\rd s \\
= \int_0^t\int_{\s^{d-1}}\int_0^1 \Big( f(\cM_{s-}+ \theta r^{1/\be(\cM_{s-})} )-f(\cM_{s-})\Big) \wn(\rd s,\rd \theta,\rd r) \\
+ \int_0^t\int_{\s^{d-1}}\!\!\int_1^{+\infty} \Big( f(\cM_{s-}+ \theta r^{1/\be(\cM_{s-})} )-f(\cM_{s-})\Big) \widetilde N(\rd s,\rd \theta,\rd r).
\end{multline*}  
Applying mean value theorem to $f$ (around zero) and \eqref{eq: linear growth+} for the first compensated Poisson integral, then $||f||_\infty<+\infty$ for the second, one concludes that 
\begin{align*}
f(M_t) - f(M_0) - \int_0^t \cL^\be f(M_s)\rd s
\end{align*}
is a martingale. By the uniqueness of the martingale problem for $(\cL^\be,C_c^2(\RR^d))$ due to Bass \cite{bass1988uniqueness}, one concludes that the solution to \eqref{theequation} is a stable-like process with index function $\be$.  The Feller property was proved in \cite[page 285]{bass1988uniqueness}. The proof is now complete.
\end{proof}

An application of \eqref{eq: linear growth+} and Burkholder-Davis-Gundy's inequality yields the following fact. 

\begin{lemma} The compensated Poisson integral in \eqref{theequation} is a martingale in $L^2(\Omega)$. 
\end{lemma}

We also need to compute the symbol of the operator $\cL^\be$ in order to use known dimension bounds for the range of L\'evy-type processes. 

\begin{lemma} The domain of $\cL^\be$ contains $C_c^{\infty}(\RR^d)$ the space of smooth functions with compact support and the restriction of $\cL^\be$ on $C_c^{\infty}(\RR^d)$ is a pseudo-differential operator with symbol $q(x,\xi)=a(x)|\xi|^{\be(x)}$ with $a: \rr^d\to\RR$ continuous bounded below and above by two positive finite constants. 
\end{lemma}
\begin{proof}
The first statement is obvious.  It remains to compute the symbol. Set 
\begin{align*}
\cF f(\xi) = \int f(x) e^{ix\cdot\xi} \rd x.
\end{align*}
For all $f\in C_c^{\infty}(\RR^d)$, by Fubini's Theorem, $\cF (\cL^\be f)(\xi) = -q(x,\xi) \cF f(\xi)$ where 
\begin{align*}
q(x,\xi) = \int (1-e^{-i u\cdot\xi} - i \xi\cdot u \indiq_{|u|\le 1}) \,u^{-d-\be(x)}\rd u.
\end{align*}
Set for $\al\in(0,2)$,
$$C_\al = \int_0^{\infty} (1-\cos r ) r^{-1-\al}\rd r. $$
Using spherical coordinate and by symmetry,
\begin{align*}
q(x,\xi) &= \int_{\s^{d-1}}\int_{0}^\infty (1-e^{ir(\theta\cdot\xi)} - i(\theta\cdot\xi) r\indiq_{r<1}) r^{-1-\be(x)}\rd r \rd\theta \\
 &= \int_{\s^{d-1}}\int_{0}^\infty (1-\cos(r(\theta\cdot\xi)) ) r^{-1-\be(x)}\rd r \rd\theta \\
 &= C_{\be(x)} \int_{\s^{d-1}} |\theta\cdot\xi|^{\be(x)}\rd\theta \\
 &= a(x)|\xi|^{\be(x)}
\end{align*}
where 
\begin{align*}
a(x)= C_{\be(x)} \int_{\s^{d-1}} |\theta\cdot \vec{e}_1|^{\be(x)}\rd\theta \quad \mbox{ with } \vec{e}_1 = (1, 0, \ldots, 0)\in \RR^d.
\end{align*}
This completes the proof.
\end{proof}

For stable-like processes with symbol as in previous lemma, Kolokoltsov \cite[Chapter 7]{kolokoltsov2011book} (see also \cite{kolokoltsov2000stablelike}) showed the existence of the transition densities and provided fine heat kernel estimates. Let us recall the part that is useful for our purposes.  
\begin{lemma}[\cite{kolokoltsov2011book}]
Let $p(t,x,y)$ be the transition density of stable-like processes with index $\be$. Let $\al= \inf_{x\in\RR^d}\be(x)$. Then there exists a finite positive $C$ so that $t<1, x,y\in \RR^d$,
\begin{align}\label{transition density}
p(t,x,y)\le C t^{-d/\al}.
\end{align}
\end{lemma}

Let us end this section with known dimension estimates for the range of L\'evy-type processes. 

\begin{lemma}[\cite{schilling1998rangefeller,knopova2015range}] \label{theobounds} Let $(X_t)_{t\ge 0}$ be a Feller process with generator $(A,\cD(A))$ such that $A_{|C_c^{\infty}(\rr^d)}$ is a pseudo-differential operator with symbol $q(x,\xi)$ satisfying  $|q(x,\xi)|\le c(1+|\xi|^2)$ for all $x$ and $q(\cdot,0)\equiv 0$. 
Then almost surely, $\dim_\cH(X[0,1])\le d\wedge \be_\infty$ where 
\begin{equation} \be_{\infty} = \inf \llb \de>0 : \lim_{|\xi|\to\infty} \frac{\sup_{|\eta|\le |\xi|}\sup_{x\in\rr^d}|q(x,\eta)|}{|\xi|^\de} = 0 \rrb .\end{equation}
If in addition the transition density of the process $X$ exists and satisfies \eqref{transition density} for some constants $C$ and $\al\in(0,2)$, then 
\begin{align*}
\dim_\mathcal{H}(X[0,1])\ge d\wedge \al \quad \mbox{ a.s. }
\end{align*}
\end{lemma}

Combining previous lemmas, it is now plain that the Hausdorff dimension for the range of stable-like processes is bounded above by $\sup_{x\in\RR^d}\be(x)\wedge
 d$ and bounded below by $\inf_{x\in\RR^d}\be(x)\wedge d$. We prove in the sequel that neither bound is optimal.

\section[Upper bound of Theorem 3.1]{Study of the $p$-variation of $\cM$ : upper bound of Theorem \ref{theorange}}

The aim of this section is to prove that
\begin{align}\label{rangeupperbound}
\dim_\cH \Big(\cM([0,1])\Big) \le \be^*_M\wedge d,  \mbox{ where } \be^*_M = \sup_{t\in[0,1]} \be(\cM_t).
\end{align}
We  use a slicing procedure for $\cM$ and the $p$-variation approach to tackle this problem. The use of $p$-variation in deducing an upper bound for the Hausdorff dimension of the range of sample paths goes back, at least, to McKean \cite{mckean1955stable}. In this article we apply a Theorem by L\'epingle \cite{lepingle1976pvariation} on the $p$-variation of semimartingales. 

First let us introduce some notations for the $p$-variation of functions.

Let $f : \rr^+ \to \rr^d$ be a c\`adl\`ag function and $\cP$ be a finite partition of the interval $[0,t]$ deduced naturally from a family of strictly ordered points $(0=t_0<\ldots < t_n = t)$. Following the notations in \cite{lepingle1976pvariation}, for any $p\in(0,2)$, let $$\ds V_p(f,\cP) = \sum_{i=0}^{n-1} |f(t_{i+1})-f(t_i)|^p.$$ Then the (strong) $p$-variation of $f$ in the interval $[0,t]$ is 
\begin{equation*} W_p(f,[0,t]) = \sup \llb V_p(f,\cP) : \cP \mbox{ finite partition of } [0,t] \rrb .
\end{equation*} 
We also introduce the quantity corresponding to the jumps of $f$ in the interval $[0,t]$,
\begin{equation*} S_p(f,[0,t]) = \sum_{0<s\le t} |\Delta f_s|^p,
\end{equation*} 
where $\Delta f_s = f(s)-f(s-)$ and $f(s-) = \lim_{t\uparrow s} f(t)$.

Recall that a semimartingale is a process of the form $X_t=X_0+ M_t+A_t$, where $X_0$ is finite a.s. and is $\cF_0$ measurable, $M_t$ is a local martingale, and $A_t$ is a process whose sample paths have bounded variation on $[0,t]$ for each $t$. Such a process can be written as $X_t = X^c_t + X^j_t$, the sum of a continuous part $X^c$ and a pure jump part $X^j$. Let us state a part of L\'epingle's result (see Theorem 1 of \cite{lepingle1976pvariation}) which is useful for our purpose.

\begin{theorem}[\cite{lepingle1976pvariation}]\label{theopvariation} Let $X$ be a semimartingale such that $\langle X^c \rangle_{\cdot} \equiv 0$. Let $p>0$. Then almost surely, 
\begin{equation*} S_p(X,[0,1])<+\infty \Longrightarrow \( \,\forall\, p'>p, \ W_{p'}(X,[0,1])<+\infty \) .
\end{equation*}
\end{theorem}

Following \cite{yang2015jumpdiffusion}, we slice the process  $\cM$ according to the different behavior of the local index process $t\mapsto\be(\cM_t)$. This induces a decomposition for the process $\cM$. Precisely, for every $m\in\nn_*$, we write $\cM_\cdot = \cM_0 + \sum_{k=0}^{m-1} \cM^{k,m}_\cdot+ \cM^{\ge 1}_\cdot$ where
\begin{equation*} \cM^{k,m}_t = \int_0^t \int_{S^{d-1}}\int_0^1 \theta r^{1/\be(\cM_\sm)}\indiq_{\be(\cM_\sm) \in [\frac{2k}{m},\frac{2k+2}{m})} \wn(ds,d\theta,dr).
\end{equation*}
and
\begin{equation*}
\cM^{\ge 1}_t = \int_0^t \int_{\s^d} \int_1^{+\infty} \theta r^{1/\be(\cM_{s-})} N(ds,d\theta, dr).
\end{equation*}
From a trajectory point of view, each sliced process behaves exactly the same as $\cM$ when the index process takes value in the sliced interval, otherwise it is only a constant process. The process $\cM^{\ge 1}$ is  not relevant  in the computation of $p$-variation for $\cM$ since it is piecewise constant with finite number of jumps in the unit interval.  

Now we are ready to prove \eqref{rangeupperbound}.  For each $(k,m)$, we study the $p$-variation of $\cM^{k,m}$, then deduce the finiteness of the $p$-variation of the whole process $\cM$ for any $p>\be^*_M$. The desired inequality follows by a general argument by McKean \cite{mckean1955stable} on the relation between Hausdorff dimension of the range of a function and its $p$-variation. 

\begin{lemma}\label{lemmapartialvariation} For every $m\in\nn_*$ and every $k= 0, \ldots, m$, almost surely,
\begin{equation*}W_{\frac{2k+3}{m}}(\cM^{k,m},[0,1])< + \infty.\end{equation*}
\end{lemma}
\begin{proof} The method consists in applying Theorem  \ref{theopvariation} to each $\cM^{k,m}$ since each of them is a semimartingale satisfying $\langle (\cM^{k,m})^c \rangle_\cdot \equiv 0$. We start with the observation that
\begin{align*} &\quad S_{(2k+\frac{5}{2})/m}(\cM^{k,m},[0,1])\\ &= \int_0^1\int_{\mathbb{S}^{d-1}} \int_0^1 r^{(2k+\frac{5}{2})/m\be(\cM_\sm)}\indiq_{\be(\cM_\sm)\in[2k/m,(2k+2)/m)}N(ds,d\theta, dr).
\end{align*} 

Taking expectation, we see that 
\begin{align*}&\E[S_{(2k+\frac{5}{2})/m}(\cM^{k,m},[0,1])] \\
&=  \int_0^1\int_0^1 \E[ r^{(2k+\frac{5}{2})/m\be(\cM_\sm)}\indiq_{\be(\cM_\sm)\in[2k/m,(2k+2)/m)}]\,\rd s\,\frac{\rd r}{r^2}  \\
&\le  \int_0^1\int_0^1 r^{(2k+\frac{5}{2})/(2k+2)-2} \,\rd r\,\rd s<+\infty.
\end{align*}
Consequently, $\ds S_{(2k+\frac{5}{2})/m}(\cM^{k,m},[0,1])$ is finite almost surely. An application of  Theorem \ref{theopvariation} ends the proof. 
\end{proof}

Now the finiteness of the $p$-variation of the whole process $\cM$ is proved as follows. 
\begin{lemma}\label{lemmaglobalvar} Almost surely, for any $p> \be^*_M$, $$W_p(\cM,[0,1])<+\infty.$$
\end{lemma}

\begin{proof} Recall that $\be^*_M$ is defined in \eqref{rangeupperbound}. Consider the events $$A_{k,m}= \llb \be^*_M + \frac{3}{m} \ge \frac{2k+3}{m} \rrb, \ \ B_{k,m} = \llb W_{\be^*_M+\frac{3}{m}}(\cM^{k,m},[0,1])< + \infty \rrb .$$ 
Since the mapping $p\mapsto \indiq_{W_p(f,[0,1])<\infty}$ is non-decreasing, one has 
$$\pp(B_{k,m} \cap A_{k,m} ) \ge \pp\( \llb W_{\frac{2k+3}{m}}(\cM^{k,m},[0,1])<+\infty\rrb \cap A_{k,m} \) .$$
Under $A^c_{k,m}$, i.e. the complementary of $A_{k,m}$, $\cM^{k,m} \equiv 0$ by the properties of the compensated Poisson integral. Hence $B_{k,m}$ is also realized. This inclusion $A^c_{k,m} \subset B_{k,m}$ yields  
$$\pp(B_{k,m}\cap A^c_{k,m}) =\pp(A^c_{k,m}).$$ 
Combining the previous two estimates, one obtains
\begin{align*} \pp(B_{k,m}) &= \pp(B_{k,m}\cap A_{k,m}) + \pp(B_{k,m}\cap A^c_{k,m}) \\
&\ge \pp(\{W_{\frac{2k+3}{m}}(\cM^{k,m},[0,1])<+\infty\}\cap A_{k,m}) + \pp(A^c_{k,m}) \\
&\ge \pp(W_{\frac{2k+3}{m}}(\cM^{k,m},[0,1])<+\infty) =1, \nonumber
\end{align*} 
where Lemma \ref{lemmapartialvariation} has been used. By Jensen's inequality (when $p\ge 1$) or subadditivity (when $p<1$), for any $p\in(0,3)$, $n\in\nn_*$ and $(a_1,\ldots,a_n)\in\rr^n$, one has $(\sum_{i=1}^n|a_i|)^p \le (n^{p-1}\vee 1) \sum_{i=1}^n |a_i|^p.$ This yields for any finite partition, every family of c\`adl\`ag functions $f_i : [0,1]\to \rr$ with $i=1,\ldots,n$, that
\begin{align*}
V_{p}\Big(\sum_{i=1}^n f_i,\cP\Big) \le C(n,p)\sum_{i=1}^n V_{p}(f_i,\cP) \le C(n,p)\sum_{i=1}^n W_{p}(f_i,[0,1]), 
\end{align*}
where $C(n,p)=n^{p-1}\vee 1$. Therefore, since $\pp(B_{k,m})=1$, one has a.s.
\begin{equation*} W_{\be^*_M+\frac{3}{m}}(\cM,[0,1])\le C(m,\be^*_M+\frac{3}{m})\sum_{k=1}^m W_{\be^*_M+\frac{3}{m}}(\cM^{k,m},[0,1])<+\infty 
\end{equation*}
for every $m\in\nn_*$, which yields the result.  
\end{proof}

\begin{preuve}{\it \ of Formula \eqref{rangeupperbound} : }
Recall the following fact in \cite{mckean1955stable} : if $f : [0,1] \to \rr^d$ is a c\`adl\`ag function with finite $p$-variation, then 
\begin{align}\dim_\cH \Big( f[0,1]\Big) \le p\wedge d. \label{variation2dimension}
\end{align}
Now \eqref{rangeupperbound} follows by combining the fact above and Lemma \ref{lemmaglobalvar}.
\end{preuve}

\section[Lower bound of Theorem 3.1]{Lower bound of Theorem \ref{theorange}}\label{sectionlower}

Throughout this section, we use $\PP^x$ to denote the law of $M$ with initial value $M_0=x\in\RR^d$. Denote by $\EE^x$ the expectation with respect to $\PP^x$.

To prove the lower bound, we introduce a suitable coupling of $\cM$ with a family of processes whose dimension of the range is known. This coupling is used in the proof of the following lemma, see \eqref{coupling} below.
\begin{lemma}\label{lemmalower} Let $0\le t_0 < 1$. For every $z\in\RR^d$, $\PP^z$-a.s.
$$\dim_\cH \Big( \cM[t_0,1]\Big) \ge \be(\cM_{t_0})\wedge d.$$
\end{lemma}

\begin{proof}   
For any $z\in\RR^d$, by the Markov property,
$$\PP^z(\dim_\cH (\cM[t_0,1])\ge \be(\cM_{t_0})| \cF_{t_0}) = g(\cM_{t_0}) \mbox{ a.s. }$$
where $$g(x) = \pp^x(\dim_\cH (M[0,1-t_0])\ge \be(x)).$$

Now one constructs a coupling with the process $\cM$.  Let $a\in(0,2)$, $x\in\RR^d$ and $\be_a(\cdot)=\be(\cdot)\vee a$. For each $\ep>0$ and any rational number $0<a\le \be(x)-2\e$, one introduces the process $\cM^{x,a}$, solution to the SDE
\begin{multline}\label{coupling} 
\cM^{x,a}_t= x + \int_0^t\int_{S^{d-1}}\int_0^1 \theta r^{1/\be_a(\cM^{x,a}_\sm)} \wn(\rd s, \rd\theta, \rd r) \\ +\int_0^t \int_{\s^d} \int_1^{+\infty} \theta r^{1/\be_a(\cM_{s-})} N(\rd s,\rd \theta, \rd r)
\end{multline}
driven by the same  Poisson random measure. 
Existence and pathwise uniqueness of these processes can be proved as in Proposition \ref{proexistence}.

Define the stopping times
\begin{align*}\tau_{x} = \inf\{t\ge 0 : \be(\cM_t)\le \be(x)-\ep \}, \\
  \tau_{x,a} = \inf\{t\ge 0 : \be(\cM^{x,a}_t)\le \be(x)-\ep \}. \nonumber
\end{align*}
Define also
\begin{align*}
\tau_{\ge 1} = \inf\{ 0\le t\le 1 :  N([0,t]\times[1,+\infty)) \ge 1 \}
\end{align*}
By the c\`adl\`ag property of the sample paths of $\cM^{x,a}$ and $\cM$, all these stopping times are strictly positive $\PP^x$-a.s..  Note that $\tau_{\ge 1}$ is an exponential random variable with finite parameter, it is also strictly positive $\PP^x$ almost surely.  Set $\tau = \min(\tau_x,\tau_{x,a}, \tau_{\ge 1})/2$.  The following observation is fundamental :  
\begin{equation}\label{claim} \PP^x \  a.s.\ \ \ \   \forall\,t\ge 0, \ \  \cM_{t\wedge \tau} = \cM^{x,a}_{t\wedge\tau}. 
\end{equation}
Indeed, for every $t\ge 0$,  using $\tau< \tau_{\ge 1}$,  one remarks that the large jump term is identically zero before time $\tau$ so that
\begin{align*}
\E^x \[ \Big|\cM^{x,a}_{t\wedge\tau}-\cM_{t\wedge\tau}\Big|^2 \] 
= \E\[ \lba \int_0^{t\wedge\tau}\int_{S^{d-1}}\int_0^1 \theta \( r^{1/\be_a(\cM^{x,a}_{s-})}-r^{1/\be(\cM_{s-})} \) \wn(ds d\theta dr)\rba ^2 \] .
\end{align*}
By Burkholder-Davis-Gundy inequality and Lipschitz continuity of $\be$,
\begin{align*}
&\E^x \[ \Big|\cM^{x,a}_{t\wedge\tau}-\cM_{t\wedge\tau}\Big|^2 \] \\
&\le C\E^x\[ \int_0^{t\wedge\tau}\int_{S^{d-1}}\int_0^1 \lba r^{1/\be_a(\cM^{x,a}_{s-})}-r^{1/\be(\cM_{s-})} \rba ^2 \frac{dr}{r^2}\,H(d\theta)\,ds \]  \\
&= C \E^x\[ \int_0^{t} \int_0^1 \lba r^{1/\be(\cM^{x,a}_{s-\wedge\tau})}-r^{1/\be(\cM_{s-\wedge\tau})} \rba ^2 \frac{dr}{r^2} ds\] \\
&\le C \E^x\[ \int_0^t |\cM^{x,a}_{s\wedge\tau}-\cM_{s\wedge\tau}|^2 ds \] \\
&= C\int_0^t \E^x [ |\cM^{x,a}_{s\wedge\tau}-\cM_{s\wedge\tau}|^2] ds, 
\end{align*}
 Hence, using Gronwall's Lemma,  for every $t\ge 0$,
\begin{equation}
\E^x\[ \lba \cM^a_{t\wedge \tau} - \cM_{t\wedge \tau} \rba ^2 \] =0.\nonumber
\end{equation}
This, along with the c\`adl\`ag property of the sample paths, yields \eqref{claim}.

To conclude, applying Lemma \ref{theobounds} (lower bound) to the stable-like  process $\cM^{x,a}$ with index function $\be_a$, we obtain that for each $t\in(0,1]$,  $$\PP^x\  a.s.\ \ \ \ \dim_\cH \cM^{x,a}([0,t])\ge \inf_{x\in\rr^d} \be_a(x) \wedge d \ge a\wedge d.$$   
This full probability set is indexed by $t$ and is non-decreasing as $t$ increases. Hence almost surely, for all $t\in (0,1]$ and all rational $a\in(1,\be(x)-2\e)$,  one has  $\dim_\cH \cM^{x,a}([0,t])\ge a\wedge d.$ One deduces that $\PP^x$ a.s. $$\dim_\cH \cM([0,1-t_0]) \ge \dim_\cH \cM([0,\tau\wedge (1-t_0)])= \dim_\cH \cM^{x,a}([0, \tau \wedge (1-t_0)]) \ge a\wedge d,$$ where we used \eqref{claim} for the equality and the fact that $\PP^x$ a.s. $\tau>0$ for the last inequality. Letting $a\to\be(x)-2\ep$ along a countable sequence, then letting $\e\to 0$,  one obtains that 
$$g(x) \equiv 1.$$
One concludes with $\pp^z(\dim_\cH \cM([t_0,1])\ge \be(\cM_{t_0})) = \E^z[g(\cM_{t_0})] = 1$. This completes the proof.

\end{proof}

Finally, we prove the lower bound in Theorem \ref{theorange}.
\begin{proof} Using Lemma \ref{lemmalower}, we have for each $t_0\in[0,1)$ that $\PP^z$ a.s. $$\dim_\cH \big(\cM([0,1])\big) \ge \dim_\cH \big( \cM([t_0,1])\big) \ge \be(\cM_{t_0})\wedge d,$$ then $\PP^z$ a.s.
$$\dim_\cH \Big(\cM([0,1])\Big) \ge \sup_{t_0\in[0,1)\cap\mathbb{Q}} \be(\cM_{t_0})\wedge d = \sup_{t_0\in[0,1]} \be(\cM_{t_0}) \wedge d, $$
where we used the c\`adl\`ag property of the sample paths. Since this holds for any $z\in\RR^d$, one deduce that the lower bound holds for any measurable $\cM_0\in\cF_0$.
\end{proof}

\section[Proof of Theorem 3.2]{Dimension of the graph of $\cM$ : proof of Theorem \ref{theograph}}

\subsection{Case $d\ge 2$}\label{secd>1} Since a projection never increases the dimension of a subset of $\rr^d$, projecting the graph on the time axis then on the space axis yields the announced lower bound for dimension of the graph. 

It remains us to prove the other inequality. Recall that $\be^*_M = \sup_{t\in[0,1]}\be(\cM_t)$. For every $p>\max(1,\be^*_M) \ge \be^*_M = \be^*_M\wedge d$, consider the $p$-variation of the process $\cG(t)=(\mbox{Id}(t),\cM_t)$ in $\rr^{d+1}$, where $\mbox{Id}(t)=t$.  As $W_p(\mbox{Id},[0,1]) \le 1$ for every $p>1$, there exists a constant $C=C(d)$ such that
\begin{align*}
W_p(\cG,[0,1]) \le C (1+W_p(\cM,[0,1])) < +\infty.
\end{align*}
by Lemma \ref{lemmaglobalvar}. Applying \eqref{variation2dimension} yields the desired upper bound.

\subsection{Case $d=1$} The proof is split into several parts. The first one gives an upper bound for the upper box-counting dimension of the graph of $\cM$, which in turn gives an upper bound for the Hausdorff dimension of the graph.  Recall that the upper box-counting dimension of a set $E\subset\RR^d$ is defined by
$$\overline{\dim_B}(E)=\limsup_{\de\downarrow 0}\frac{\log N_\de(E)}{-\log\de},$$
where $N_{\de}(E)$ is the smallest number of sets of diameter at most $\de$ to cover $E$, 
 see Chapter 3 of  \cite{falconer2003}.  The proof is quite standard, see for instance \cite{norvaisa2002pvariation}. We prove it for completeness.  

\begin{proposition}\label{lemmauppergraph} Almost surely, $$\dim_\cH(Gr_{[0,1]}(\cM))\le \overline{\dim_B}(Gr_{[0,1]}(\cM))\le \max\( 1,2-\frac{1}{\be_M^*}\) .$$
\end{proposition} 
\begin{proof}
The left inequality is a general fact, see \cite[page 46]{falconer2003}. Let us prove   the right inequality. 

If the event $\{ \be_M^*<1\}$ is realized, Lemma \ref{lemmaglobalvar} yields that the process $\cM$ has finite variation, {\it a fortiori}, the graph process $\cG$ has finite variation. Hence the dimension of the graph of $\cM$ (which is the range of $\cG$) is $1$  by the projection argument used in Section \ref{secd>1} and  \eqref{variation2dimension}. The desired inequality is straightforward.

If $\{\be_M^*\ge 1\}$ is realized, we consider $p>\be_M^*\ge 1$ and relate the upper box dimension with the $p$-variation of the process. Denote the oscillation of the process $\cM$ in the dyadic interval $[k2^{-j}, (k+1)2^{-j}]$ by $$Osc(\cM,I_{j,k}) := \sup \{ |\cM_s -\cM_t| : s,t\in I_{j,k} \}.$$ For every $k$, $Gr_{I_{j,k}}(\cM)$ can be covered by at most $2^j Osc(\cM,I_{j,k})+2$ squares of side length $2^{-j}$. The number of squares of generation $j$ required to cover the graph $Gr_{[0,1]}(\cM)$ satisfies
\begin{align*}
N_j &= \sum_{k=0}^{2^j-1} \big( 2^j Osc(\cM,I_{j,k})+2\big) \le 2^j \sum_{k=0}^{2^j-1} W_p(\cM,I_{j,k})^{\frac{1}{p}}+ 2\cdot 2^j \\ 
&\le 2^j \( \sum_{k=0}^{2^j-1}W_p(\cM,I_{j,k}) \) ^{\frac{1}{p}} (2^j)^{1-\frac{1}{p}} + 2\cdot 2^j \\ 
&\le  2\cdot 2^{j(2-\frac{1}{p})} W_p(\cM,[0,1])^{\frac{1}{p}}
\end{align*}
for all $j$ large enough, where we used H\"older inequality for the second inequality. Therefore, $$\ \overline{\dim_B}(Gr_{[0,1]}(\cM)) \le \limsup_{j\to\infty} \frac{\log N_j}{\log 2^j} \le 2- \frac{1}{p},$$
where we used $W_p(\cM,[0,1])<\infty$. Letting $p\to \be_M^*$ yields the result.
\end{proof}

\sk
The rest of this section is devoted to prove the lower bound in Theorem \ref{theograph}, 
\begin{align*}
\dim_\cH\big(Gr_{[0,1]}(\cM) \big) \ge 1\vee \Big( 2-\frac{1}{\be_M^*} \Big).
\end{align*} 
To prove it, we give a deterministic lower bound for the dimension of graph.  This should be viewed as an analogue of Lemma \ref{theobounds} (the lower bound part) in the graph context.

\begin{proposition}\label{lemmalowergraph} Denote $\underline{\be}=\inf_{x\in\rr^d}\be(x)$. Almost surely,  
\begin{equation}\label{eq:prop3} \dim_\cH\big(Gr_{[0,1]}(\cM) \big) \ge 1\vee \big( 2-\frac{1}{\underline{\be}}\big).
\end{equation}
\end{proposition}
We prove this proposition in several steps. First we adapt the ideas in \cite{pruitt1969stablecomponents} to give tail estimates for the sojourn time of $\cM$. This allows to understand the local behavior of the graph occupation measure. Then we use a density argument to obtain the lower bound for the Hausdorff dimension of the graph of $\cM$.  

Following Pruitt-Taylor \cite{pruitt1969stablecomponents}, we define the sojourn time of $\cM$ in the ball centered at $\cM_{t_0}$ with radius $a>0$, during the time interval $[t_0,t_0+s]$ for $0<s<1-t_0$ as
\begin{align*}
T_{t_0}(a,s) = \int_{t_0}^{t_0+s} \indiq_{|\cM_t-\cM_{t_0}|\le a} dt. 
\end{align*}
Write for simplicity $T(a,s)=T_0(a,s)$. The main estimate is the following.
\begin{lemma}\label{taillemma} Fix $t_0\in[0,1)$. Assume that $\underline{\be}>1$ and let $C=\frac{2}{1-1/\underline{\be}}$. For every $0<s\le 1-t_0$, $\la>0$, $a>0$, one has
\begin{align*}
\pp(T_{t_0}(a,s)\ge \la a s^{1-\frac{1}{\underline{\be}}}) \le e^{-\la/2C}
\end{align*}
\end{lemma}
\begin{proof}
Recall the notations $\PP^x$, $\EE^x$ in the proof of Lemma \ref{lemmalower}. By the Markov property,
\begin{align*}
\PP(T_{t_0}(a,s)\ge \la a s^{1-\frac{1}{\underline{\be}}}|\cF_{t_0})=g(M_{t_0}) \mbox{ a.s. }
\end{align*}
where $g(x)=\PP^x(T(a,s)\ge \la a s^{1-\frac{1}{\underline{\be}}})$. It suffices to prove that the upper tail estimate
\begin{align*}
\pp^x(T(a,s)\ge \la a s^{1-\frac{1}{\underline{\be}}}) \le e^{-\la/2C}
\end{align*}
holds uniformly for all $x\in\RR^1$. To do so, let us compute the moment generating function of the sojourn time. First, we study its $n$-th moment for all $n\ge 2$. For every $k\in\nn$ and $s\in\rr^+$, let $$\Gamma_k=\Gamma_k(s)=\{(t_1,\ldots,t_k)\in [0,s]^k : 0\le t_1\le \cdots \le t_k\le s \}.$$ Applying Fubini Theorem,
\begin{align*}
&\E^x[T(a,s)^n]\\
&= \int_0^s\cdots\int_0^s \pp^x\big( \bigcap_{i=1}^n \{|\cM_{t_i}|\le a\} \big) \rd t_1\cdots \rd t_n \\
&= n! \int_{\Gamma_n} \pp^x\( \bigcap_{i=1}^{n} \{|\cM_{t_i}|\le a\} \) \rd t_1\cdots \rd t_{n} \\
&\le n! \int_{\Gamma_n} \pp^x\( \bigcap_{i=1}^{n-1} \{|\cM_{t_i}|\le a\}, |\cM_{t_n}-\cM_{t_{n-1}}|\le 2a \) \rd t_1\cdots \rd t_n\\ 
\end{align*}
which, by the Markov property, is equal to
\begin{multline*}
n! \int_{\Gamma_n} \E^x\[ \pp^x\( \bigcap_{i=1}^{n-1} \{|\cM_{t_i}|\le a\}| \cF_{t_{n-1}} \)
\pp^{\cM_{t_{n-1}}}\( |M_{t_n} -M_{t_{n-1}}|\le 2a \) \] 
\rd t_1\cdots \rd t_n.
\end{multline*}
Integrating over $t_n$, then pull out the conditional first moment of the sojourn time, one gets the upper bound 
\begin{align*}
&\E^x[T(a,s)^n]\\
&\le  n! \int_{\Gamma_{n-1}} \E^x\[ \pp^x\( \bigcap_{i=1}^{n-1} \{|\cM_{t_i}|\le a\}| \cF_{t_{n-1}}  \) \E^{\cM_{t_{n-1}}}[T(2a,s-t_{n-1})] \] \  dt_1\cdots dt_{n-1} \\
&\le n\cdot \( \sup_{x\in\rr^d}\E^x[T(2a,s)]\) \cdot (n-1)!\int_{\Gamma_{n-1}} \hspace{-3mm}\pp^x\( \bigcap_{i=1}^{n-1} \{|\cM_{t_i}|\le a\} \) dt_1\cdots dt_{n-1} \\
&=  n\cdot \( \sup_{x\in\rr^d}\E^x[T(2a,s)]\) \cdot \E^x[T(a,s)^{n-1}].
\end{align*}
We iterate this procedure to get 
\begin{align*}
\E^x[T(a,s)^n] \le n! \( \sup_{x\in\rr^d} \E^x[T(2a,s)] \) ^n.
\end{align*}
Thus for all $u> 0$, the exponential moment of $T(a,s)$ is  bounded from above by
\begin{align}\label{eq:lem8}
\E^x[e^{uT(a,s)}] = \sum_{n=0}^{+\infty} \frac{u^n}{n!}\E^x[T(a,s)^n] \le \sum_{n=0}^{+\infty} \( u \sup_{x\in\rr^d} \E^x[T(2a,s)] \) ^n.
\end{align}
 Applying  the density estimate \eqref{transition density} yields for all $x\in\rr^d$, $0<s<1$,
\begin{align}\label{eq:lem8_1}
\E^x [T(2a,s)] = \int_0^s \pp^x (|M_t-x|\le 2a ) dt \le \int_0^s t^{-1/\underline{\be}}\cdot 2a \,dt \le C \cdot as^{1-\frac{1}{\underline{\be}}}
\end{align}
with $C=\frac{2}{1-1/\underline{\be}}$, recalling that here $d=1$.
Finally we choose $$ u=\frac{1}{2\sup_{x\in\rr^d} \E^x[T(2a,s)]}$$ 
so that the exponential moment \eqref{eq:lem8} is bounded above by $1$.  Consequently, using the bound \eqref{eq:lem8_1} and the Markov inequality yields that
\begin{align*}
\pp^x(T(a,s)\ge \la a s^{1-\frac{1}{\underline{\be}}}) &\le e^{-u\la as^{1-\frac{1}{\underline{\be}}}} \E^x[e^{uT(a,s)}] \le e^{-\frac{\la}{2C}}.
\end{align*}
Since the estimate is uniform in $x$, the proof is complete. 
\end{proof}

The following density lemma is useful for our purpose. Comparing to the usual mass distribution principle (\cite[page 60]{falconer2003}),  it is the semi-time interval $[t,t+h]$ that is used rather than $[t-h,t+h]$ in order to apply the Markov property. We refer to \cite[Lemma 4]{taylor1966} for a proof.  Recall that the Hausdorff measure of $E$ related to the gauge function $\ph$ is defined by $\cH^\ph(E)=\lim_{\de\downarrow 0}\cH^\ph_\de(E)$ where
\begin{align*}
\cH^\ph_\de(E)= \inf\left\lbrace\sum_{i} \ph({\rm diam}(Q_i)): E\subset\bigcup_{i} Q_i \mbox{ with } {\rm diam}(Q_i)\le \de\right\rbrace
\end{align*}
and $\ph:\RR_+\to\RR_+$ is an increasing function satisfying $\ph(2x)\le K\ph(x)$ around zero for some finite constant $K$.

\begin{lemma}[\cite{taylor1966}]\label{densitylemma} Suppose that $\nu$ is a probability measure supported on $E\subset[0,1]\times\rr$ such that for $\nu$-almost every $(t,x)$,
$$\limsup_{h\to 0} \frac{\nu\( [t,t+h]\times[x-h,x+h]\) }{\ph(h)}\le C <+\infty.$$
Then $$\cH^{\ph}(E)\ge \frac{1}{C}.$$
\end{lemma}

\begin{preuve}{\it \ of Proposition \ref{lemmalowergraph} : } 
Observe that the right hand side term in \eqref{eq:prop3} is $1$ when $\underline{\be}\le 1$. Using again  the fact that the projection of a set in $\RR^2$ to any line does not increase the Hausdorff dimension, we see that $\dim_\cH(Gr_{[0,1]}(\cM))\ge \dim_\cH([0,1])=1$, as desired.

Now consider $\underline{\be}>1$. For any $t_0\in[0,1)$,  Lemma \ref{taillemma} applied to $a = s= 2^{-m}$ and $\la= m$ so that $2^{-m}<1-t_0$ yields that
\begin{align*}
\pp\Big(T_{t_0}(2^{-m},2^{-m})\ge m 2^{-m(2-1/\underline{\be})}\Big) \le e^{-m/2C}.
\end{align*} 
We deduce using the Borel-Cantelli Lemma that a.s. for all $m$ large enough,
$$T_{t_0}(2^{-m},2^{-m})\le m 2^{-m(2-1/\underline{\be})}.$$ For all $a$ small enough, let $m$ be the unique integer such that $2^{-m-1}\le a<2^{-m}$. Then 
$$\frac{T_{t_0}(a,a)}{\log(1/a)a^{2-1/\underline{\be}}}\le \frac{T_{t_0}(2^{-m},2^{-m})}{(\log 2)m2^{-m(2-1/\underline{\be})}}   \frac{2^{-m(2-1/\underline{\be})}}{2^{-(m+1)(2-1/\underline{\be})}} \le C$$
where $C$ is a positive finite constant independent of $m$. Thus for any $t_0\in[0,1)$,  a.s.
\begin{align}\label{localexpoofmu}
\limsup_{a\to 0} \frac{T_{t_0}(a,a)}{\log(1/a)a^{2-1/\underline{\be}}}\le C.
\end{align}
Consider the probability measure $\mu$, defined by $\ds \mu(A):= \int_0^1 \indiq_A(t,\cM_t) dt$ whose support is  the graph $Gr_{[0,1]}(\cM)$. The estimate  \eqref{localexpoofmu}  yields that for any fixed $t_0\in[0,1)$,
\begin{align*}
\limsup_{a\to 0}\frac{\mu([t_0,t_0+a]\times[\cM_{t_0}-a,\cM_{t_0}+a])}{\log(1/a)a^{2-1/\underline{\be}}} \le C \quad \mbox{ a.s. }
\end{align*}  
A Fubini argument yields that a.s. 
\begin{align}\label{eq:pro3_1}
\mbox{ for } \mbox{Lebesgue a.e. } t\in(0,1), \, \qquad \limsup_{a\to 0}\frac{\mu([t,t+a]\times[\cM_t-a,\cM_t+a])}{\log(1/a)a^{2-1/\underline{\be}}} \le C 
\end{align}
Denote by $\mathcal{N}\subset[0,1]$ the Lebesgue null set such that \eqref{eq:pro3_1} fails and set $$G_{\mathcal N}= \bra{(t,M_t)\in Gr_{[0,1]}(M): t\in \mathcal N},$$ then $$\mu\left(G_{\mathcal N}\right)= \int_0^1 1_{G_{\mathcal N}}(t,M_t)\,dt=0.$$ 
This, together with Lemma \ref{densitylemma} applied to $\mu$, yields that a.s. $\cH^\ph(Gr_{[0,1]}(\cM))\ge 1/C$ with $\ph(x)=\log(1/x) x^{2-1/\underline{\be}}$. The desired lower bound for the Hausdorff dimension of  $Gr_{[0,1]}(\cM)$ follows. 
\end{preuve}

\sk
\sk
Finally we prove Theorem \ref{theograph} when $d=1$.
\begin{proof} 
The upper bound is deduced from Proposition \ref{lemmauppergraph}. Let us show the lower bound. To do so, we claim that for every $z\in\RR$ and $t_0\in[0,1)$, $\PP^z$ a.s., $$\dim_\cH Gr_{[t_0,1]}(\cM)\ge \max\Big( 1, 2-\frac{1}{\be(\cM_{t_0})}\Big) .$$
Consequently, $\PP^z$ a.s.
\begin{align*}
\dim_\cH \Big( Gr_{[0,1]}(\cM)\Big) \ge \sup_{t_0\in [0,1]\cap \mathbb{Q}} \max\Big( 1,2-\frac{1}{\be(\cM_{t_0})} \Big) \\
= \max\Big( 1, 2-\frac{1}{\sup_{t\in[0,1]}\be(\cM_t)} \Big) .
\end{align*}
As the lower bound holds uniformly in $z\in\RR$, the result follows.

It remains to prove the claim. Using  the Markov property as in beginning of the proof of Lemma \ref{lemmalower}, it suffices to show that for any $x\in\RR$, 
\begin{align}\label{eq:last_0}
\PP^x(\dim_\cH Gr_{[0,1-t_0]}(\cM)\ge 1\vee(2-1/\be(x)))=1.
\end{align}
To this end, fix $x\in\RR$ and consider the family of processes $\{M^{x,a}; a<\be(x)\}$ constructed in \eqref{coupling}. The  property \eqref{claim} satisfied by $M^{x,a}$, together with Proposition \ref{lemmalowergraph} applied to the stable-like process with index function $\be(\cdot)\vee a$, immediately implies \eqref{eq:last_0} by letting $a\to \be(x)$. The proof is now complete.

\end{proof}

\section{Discussion}
This paper deals with the a typical family of L\'evy-type processes with variable order symbol. The SDE techniques used here allow to improve previously deterministic dimension bound to a stochastic one, and in the case of stable-like processes, the new bound is actually optimal.  It would be interesting to see whether the SDE point of view allows to  get dimension bounds for more general L\'evy-type processes, in particular for those that do not have a density estimate like \eqref{transition density}.

One possible extension of this article is the study of $\dim_\cH\cM(E)$ with $E$ being any Borel set in $\rr^+$. In \cite{knopova2015range}, this question was considered and the authors obtained some bounds. The slicing and coupling argument of the present paper may certainly improve the bounds obtained in \cite{knopova2015range}.  In one dimension, under monotonicity assumptions, quite precise answer to this question is given in \cite{SeuretYang17}.

\appendix
\section*{Appendix}
Here we close the gap in the proof of Proposition \ref{proexistence}.
As the range of $\be(\cdot)$ is included in a compact set of $(0,2)$, there exists $\e>0$ such that $x\mapsto \be(x)$ is uniformly bounded from above by $2-\e$. Hence, 
\begin{align*}
 \int_{\s^{d-1}}\int_0^1 |\theta r^{1/\be(x)}|^2 \, \frac{\rd r}{r^2}\,H(\rd \theta)  &=  \int_0^1 r^{\frac{2}{2-\ep}} \,\frac{\rd r}{r^2} := C<+\infty.
\end{align*}
The growth condition is thus satisfied.

Let us now consider the Lipschitz condition. Let $x,y\in\rr^{d}$. Without loss of generality, we assume $\be(x)>\be(y)$; then
\begin{align*}
&\int_{\s^{d-1}}\int_0^1 |\theta r^{1/\be(x)}-\theta r^{1/\be(y)}|^2 \, \frac{dr}{r^2}\,H(d\theta) \\
&= \int_0^1 (r^{1/\be(x)}-r^{1/\be(y)})^2 \, \frac{dr}{r^2} \\
&= \int_0^1 r^{2/\be(x)} \( 1- e^{(\log \frac{1}{r}) \( \frac{1}{\be(x)}- \frac{1}{\be(y)}\) }\) ^2 \,\frac{dr}{r^2}.
\end{align*}
Using the inequality $1-e^{-u}\le u$ for $u>0$, this integral is bounded above by 
\begin{align*}
&\int_0^1 r^{2/\be(x)}  \( \log \frac{1}{r}\) ^2 \( \frac{1}{\be(x)}- \frac{1}{\be(y)}\) ^2  \,\frac{dr}{r^2} \\
&\le C|x-y|^2 \int_0^1 r^{2/\be(x)} (\log\frac{1}{r})^2\,\frac{dr}{r^2},
\end{align*}
where we used the Lipschitz continuity of the function $\be$. Remark that $\log(1/r)^2 \le C r^{-\ep_0}$ for every $r\in(0,1)$ where $\ep_0 = \frac{1}{2}(\frac{2}{\sup_{x\in\rr^d}\be(x)}-1)$. Hence the last integral is finite and independent of $(x,y)$. The Lipschitz condition follows.

\section*{Acknowledgements}
This work is part of my phd thesis at Universit\'e Paris-Est. I wish to thank my advisors St\'ephane Jaffard and St\'ephane Seuret for their constant support during the preparation of this paper. I also wish to thank Nicolas Fournier for stimulating discussions. I thank the anonymous referee for her/his careful reading and remarks which improve the quality of the manuscript.


\bibliographystyle{plain}
\bibliography{xyangbiblio}

\begin{thebibliography}{10}

\bibitem{bass1988uniqueness}
Richard~F. Bass.
\newblock Uniqueness in law for pure jump {M}arkov processes.
\newblock {\em Probab. Theory Related Fields}, 79(2):271--287, 1988.

\bibitem{blumenthal1960stable}
R.~M. Blumenthal and R.~K. Getoor.
\newblock Some theorems on stable processes.
\newblock {\em Trans. Amer. Math. Soc.}, 95:263--273, 1960.

\bibitem{blumenthal1962zeroANDgraph}
R.~M. Blumenthal and R.~K. Getoor.
\newblock The dimension of the set of zeros and the graph of a symmetric stable
  process.
\newblock {\em Illinois J. Math.}, 6:308--316, 1962.

\bibitem{bottcher2013survey}
Bj{\"o}rn B{\"o}ttcher, Ren{\'e} Schilling, and Jian Wang.
\newblock {\em L\'evy matters. {III}}, volume 2099 of {\em Lecture Notes in
  Mathematics}.
\newblock Springer, Cham, 2013.
\newblock L{\'e}vy-type processes: construction, approximation and sample path
  properties, With a short biography of Paul L{\'e}vy by Jean Jacod, L{\'e}vy
  Matters.

\bibitem{falconer2003}
Kenneth Falconer.
\newblock {\em Fractal geometry}.
\newblock John Wiley \& Sons, Inc., Hoboken, NJ, second edition, 2003.
\newblock Mathematical foundations and applications.

\bibitem{fu2010positiveSDE}
Zongfei Fu and Zenghu Li.
\newblock Stochastic equations of non-negative processes with jumps.
\newblock {\em Stochastic Process. Appl.}, 120(3):306--330, 2010.

\bibitem{jacob2005vol3}
N.~Jacob.
\newblock {\em Pseudo differential operators and {M}arkov processes. {V}ol.
  {III}}.
\newblock Imperial College Press, London, 2005.
\newblock Markov processes and applications.

\bibitem{jain1968graph}
Naresh Jain and William~E. Pruitt.
\newblock The correct measure function for the graph of a transient stable
  process.
\newblock {\em Z. Wahrscheinlichkeitstheorie und Verw. Gebiete}, 9:131--138,
  1968.

\bibitem{khoshnevisan2003range}
Davar Khoshnevisan, Yimin Xiao, and Yuquan Zhong.
\newblock Measuring the range of an additive {L}\'evy process.
\newblock {\em Ann. Probab.}, 31(2):1097--1141, 2003.

\bibitem{knopova2015range}
V.~Knopova, R.~L. Schilling, and J.~Wang.
\newblock Lower bounds of the {H}ausdorff dimension for the images of {F}eller
  processes.
\newblock {\em Statist. Probab. Lett.}, 97:222--228, 2015.

\bibitem{kolokoltsov2000stablelike}
Vassili Kolokoltsov.
\newblock Symmetric stable laws and stable-like jump-diffusions.
\newblock {\em Proc. London Math. Soc. (3)}, 80(3):725--768, 2000.

\bibitem{kolokoltsov2011book}
Vassili~N. Kolokoltsov.
\newblock {\em Markov processes, semigroups and generators}, volume~38 of {\em
  de Gruyter Studies in Mathematics}.
\newblock Walter de Gruyter \& Co., Berlin, 2011.

\bibitem{lepingle1976pvariation}
D.~L{\'e}pingle.
\newblock La variation d'ordre {$p$} des semi-martingales.
\newblock {\em Z. Wahrscheinlichkeitstheorie und Verw. Gebiete},
  36(4):295--316, 1976.

\bibitem{mckean1955stable}
Henry~P. McKean, Jr.
\newblock Sample functions of stable processes.
\newblock {\em Ann. of Math. (2)}, 61:564--579, 1955.

\bibitem{norvaisa2002pvariation}
Rimas Norvai{\v{s}}a and Donna~Mary Salopek.
\newblock Estimating the {$p$}-variation index of a sample function: an
  application to financial data set.
\newblock {\em Methodol. Comput. Appl. Probab.}, 4(1):27--53, 2002.

\bibitem{pruitt1969stablecomponents}
W.~E. Pruitt and S.~J. Taylor.
\newblock Sample path properties of processes with stable components.
\newblock {\em Z. Wahrscheinlichkeitstheorie und Verw. Gebiete}, 12:267--289,
  1969.

\bibitem{pruitt1969range}
William~E. Pruitt.
\newblock The {H}ausdorff dimension of the range of a process with stationary
  independent increments.
\newblock {\em J. Math. Mech.}, 19:371--378, 1969/1970.

\bibitem{schilling1998rangefeller}
Ren{\'e}~L. Schilling.
\newblock Feller processes generated by pseudo-differential operators: on the
  {H}ausdorff dimension of their sample paths.
\newblock {\em J. Theoret. Probab.}, 11(2):303--330, 1998.

\bibitem{SeuretYang17}
St\'ephane Seuret and Xiaochuan Yang.
\newblock Multifractal analysis for the occupation measure of stable-like
  processes.
\newblock {\em Electron. J. Probab.}, 22:1--36, 2017.

\bibitem{situ2005}
Rong Situ.
\newblock {\em Theory of stochastic differential equations with jumps and
  applications}.
\newblock Mathematical and Analytical Techniques with Applications to
  Engineering. Springer, New York, 2005.
\newblock Mathematical and analytical techniques with applications to
  engineering.

\bibitem{taylor1953}
S.~J. Taylor.
\newblock The {H}ausdorff {$\alpha$}-dimensional measure of {B}rownian paths in
  {$n$}-space.
\newblock {\em Proc. Cambridge Philos. Soc.}, 49:31--39, 1953.

\bibitem{taylor1966}
S.~J. Taylor and J.~G. Wendel.
\newblock The exact {H}ausdorff measure of the zero set of a stable process.
\newblock {\em Z. Wahrscheinlichkeitstheorie und Verw. Gebiete}, 6:170--180,
  1966.

\bibitem{tsuchiya1992}
Masaaki Tsuchiya.
\newblock L\'evy measure with generalized polar decomposition and the
  associated {SDE} with jumps.
\newblock {\em Stochastics Stochastics Rep.}, 38(2):95--117, 1992.

\bibitem{xiao2004survey}
Yimin Xiao.
\newblock Random fractals and {M}arkov processes.
\newblock In {\em Fractal geometry and applications: a jubilee of {B}eno\^\i t
  {M}andelbrot, {P}art 2}, volume~72 of {\em Proc. Sympos. Pure Math.}, pages
  261--338. Amer. Math. Soc., Providence, RI, 2004.

\bibitem{yang2015jumpdiffusion}
Xiaochuan Yang.
\newblock Multifractality of jump diffusion processes.
\newblock {\em Arxiv, e-print}, 2015.

\end{thebibliography}

\end{document}